\documentclass[microtype]{jloganal}
\usepackage[utf8]{inputenc}
\usepackage{amsmath}
\usepackage{amsthm}
\usepackage{amssymb}
\usepackage{mathabx}
\usepackage[only,llbracket,rrbracket]{stmaryrd}

\newtheorem{theorem}{Theorem}
\newtheorem{lemma}[theorem]{Lemma}

\newtheorem{example}[theorem]{Example}
\newtheorem{remark}[theorem]{Remark}
\newtheorem{proposition}[theorem]{Proposition}
\newtheorem{corollary}[theorem]{Corollary}
\newtheorem{definition}[theorem]{Definition}

\newtheorem{convention}[theorem]{Convention}

\def\RM{\mathrm{rm}}

\title[Hyperreal differentiation with an idempotent ultrafilter]{Hyperreal differentiation with an idempotent ultrafilter}

\author[S\,A Alexander]{Samuel Allen Alexander}
\givenname{Samuel Allen}
\surname{Alexander}
\address{The U.S.\ Securities and Exchange Commission}
\email{samuelallenalexander@gmail.com}

\author[B Dawson]{Bryan Dawson}
\givenname{Bryan}
\surname{Dawson}
\address{Union University}
\email{bdawson@uu.edu}


\volumenumber{}
\issuenumber{}
\publicationyear{}
\papernumber{}
\startpage{}
\endpage{}
\doi{}
\MR{}
\Zbl{}
\received{}
\revised{}
\accepted{}
\published{}
\publishedonline{}
\proposed{}
\seconded{}
\corresponding{}
\editor{}
\version{}

\begin{abstract}
    In the hyperreals constructed using a free ultrafilter on $\mathbb R$,
    where $[f]$ is the hyperreal represented by $f:\mathbb R\to\mathbb R$,
    it is tempting to define a derivative operator by $[f]'=[f']$, but
    unfortunately this is not generally well-defined. We show that if the
    ultrafilter in question is idempotent and contains $(0,\epsilon)$ for
    arbitrarily small real $\epsilon$ then the desired derivative
    operator is well-defined for all $f$ such that $[f']$ exists.
    We also introduce a hyperreal variation of the derivative from finite
    calculus, and show that it has surprising relationships to the
    standard derivative. We give an alternate proof,
    and strengthened version of, Hindman's theorem.
\end{abstract}

\begin{asciiabstract}
    In the hyperreals constructed using a free ultrafilter on R,
    where [f] is the hyperreal represented by f:R->R,
    it is tempting to define a derivative operator by [f]'=[f'], but
    unfortunately this is not generally well-defined. We show that if the
    ultrafilter in question is idempotent and contains (0,epsilon) for
    arbitrarily small real epsilon then the desired derivative
    operator is well-defined for all f such that [f'] exists.
    We also introduce a hyperreal variation of the derivative from finite
    calculus, and show that it has surprising relationships to the
    standard derivative. We give a new proof of Hindman's theorem, and
    we prove a stronger theorem.
\end{asciiabstract}

\begin{document}

\maketitle

\section{Introduction}

There is a long tradition
\cite{shelly1911cuestion,barbeau1961remarks,ufnarovski2003differentiate,
buium2005arithmetic,stay2005generalized,kovic2012arithmetic,buium2015differential,pasten2022arithmetic,buium2023foundations} of attempting to differentiate numbers in various ways.
Much attention was focused on derivatives of numbers when Jeffries'
paper on the subject appeared in the Notices of the
AMS late last year \cite{jeffries1772differentiating}; almost simultaneously
(and apparently independently),
Tossavainen et al's survey on the subject appeared in the College Math Journal
\cite{tossavainen2024we}.

Why should the reader care about differentiating numbers?
In general, any time a new theory is introduced, it is natural to seek
\emph{numerical} structures satisfying that theory: thus, when the theory of groups
is introduced, it is natural to introduce examples like $(\mathbb Q,+)$
and $(\mathbb R^+,\cdot)$.
Theories about elementary calculus functions, in languages
including the unary function symbol $\bullet'$, were originally modeled
by structures whose universes consisted of elementary calculus functions,
\emph{not} numbers. We can at least try to find numerical models for these
theories. The act of interpreting
$\bullet'$ in a structure whose universe is a number system
is the act of ``differentiating numbers''. We are hopeful that
generalizing models of elementary calculus function theories could eventually
be fruitful just like generalizing models of permutation sets led to
the abstract theory of groups.

When it comes to numerically modeling theories,
different theories might require different
number systems: both $(\mathbb Z,+,\cdot)$ and $(\mathbb Q,+,\cdot)$
are rings, but only the latter is a field.
By gaining knowledge about which number systems are needed for which theories,
we gain insight into those theories. We hope that a greater
knowledge about which number systems are needed to model various subtheories of
elementary calculus functions, will eventually give us insight into
those subtheories.

If the only axiom we care about is the
Leibniz rule (and the nontriviality axiom $\exists x\,\mbox{s.t.}\,x'\not=0$),
we can interpret $\bullet'$
on $\mathbb N$ so as to satisfy that. That is the approach of
\cite{barbeau1961remarks,ufnarovski2003differentiate}.
But their $(\mathbb N,\bullet')$ does not even satisfy the
linearity axiom. Our interpretation (in Section \ref{mainsection})
of $\bullet'$ on a subset of the hyperreals
will satisfy far more axioms of the theory of elementary calculus functions.
And in Section \ref{entirenumberssectn},
we will introduce a stricter subset of the hyperreals where
not only $\bullet'$ can be elegantly interpreted, but $\circ$ as well
(in other words, there is an elegant way to define the ``composition'' of
two numbers there),
in such a way as to numerically satisfy even the chain rule.

The key idea behind our numerical interpretation of $\bullet'$ is to
commute the derivative
operation with the operation of taking a function's equivalence class in the
hyperreals, in other words, define $[f]'=[f']$ (we will spell out the details below). Unfortunately, this is not well-defined in general. However, the idea
can be salvaged in several different ways, by making use of certain idempotent ultrafilters. In Section \ref{finitecalculussection} we will use the same
approach to well-define a hyperreal variation of the derivative from finite
calculus, and as an application of that, we will give a new proof of Hindman's
theorem and also strengthen said theorem.

\section{Preliminaries}

Throughout the paper, we write $\beta\mathbb R$ for the set of ultrafilters on
$\mathbb R$.

\begin{definition}
\label{hyperrealsdefn}
(Hyperreals)
For each free $p\in\beta\mathbb R$, let ${}^*\mathbb R_p$
be the hyperreals constructed using $p$. For every $f:\mathbb R\to\mathbb R$,
let $[f]_p$ be the hyperreal represented by $f$. If $p$ is clear from
context, we will write ${}^*\mathbb R$ and $[f]$ for ${}^*\mathbb R_p$
and $[f]_p$, respectively.
\end{definition}

\begin{convention}
\label{undefdconvention}
If $p\in\beta\mathbb R$ is free and $f$ is a function with codomain
$\mathbb R$ and with domain $\mathrm{dom}(f)\in p$, we will write
$[f]_p$ for $[\hat f]_p$ where $\hat f:\mathbb R\to\mathbb R$
is the extension of $f$ defined by
$\hat f(x)=0$ for all $x\in\mathbb R\backslash\mathrm{dom}(f)$.
If $\mathrm{dom}(f)\not\in p$, we say that $[f]_p$ \emph{does not exist}.
If $p$ is clear from context, we will write $[f]$ for $[f]_p$.
\end{convention}

\begin{definition}
    For each $f:\mathbb R\to\mathbb R$, let ${}^*f:{}^*\mathbb R\to{}^*\mathbb R$
    be the nonstandard extension of $f$. Let $\Omega=[x\mapsto x]$ be
    the hyperreal represented by the identity function.
\end{definition}

For every $f:\mathbb R\to\mathbb R$, there are two ways of viewing $f$
in nonstandard analysis. It can be viewed as the number $[f]$ or as the
function ${}^*f:{}^*\mathbb R\to{}^*\mathbb R$. The two are related via
$\Omega$. Namely: $[f]={}^*f(\Omega)$.

Unfortunately, the following proposition shows that the idea of defining
$[f]'=[f']$ does not work in general.

\begin{proposition}
\label{illdefdprop}
    (Ill-definedness)
    \begin{enumerate}
        \item
        There exists a free $p\in\beta\mathbb R$
        and everywhere-differentiable $f,g:\mathbb R\to\mathbb R$
        such that $[f]_p=[g]_p$ but $[f']_p\not=[g']_p$.
        \item
        For every free $p\in\beta\mathbb R$,
        for all $f:\mathbb R\to\mathbb R$ such that $[f']_p$ exists,
        there exists $g:\mathbb R\to\mathbb R$ such that $[f]_p=[g]_p$
        but $[g']_p$ does not exist.
    \end{enumerate}
\end{proposition}

\begin{proof}
    (1) Let $p\in\beta\mathbb R$ such that
    $\mathbb N\in p$. The claim is witnessed by $f(x)=0$
    and $g(x)=\sin \pi x$.

    (2) Let $D\subseteq \mathbb R$ be dense and co-dense.
    Assume $D\in p$ (if not, then $D^c\in p$ and a similar argument applies).
    The claim is witnessed by $g(x)=f(x)+\chi_{D^c}(x)$.
\end{proof}

In light of Proposition \ref{illdefdprop}, we cannot expect the definition
$[f]'_p=[f']_p$ to work for every free $p\in\beta\mathbb R$ even
if we restrict our attention to everywhere-differentiable $f$; and if we
do not so restrict our attention, then we can \emph{expect} the definition
$[f]'_p=[f']_p$ to fail for every $p$.
We will show that if we restrict attention to those $f$ such that $[f']$
exists, then the definition $[f]'_p=[f']_p$ \emph{does} work provided
$p$ is idempotent and contains $(0,\epsilon)$ for every $\epsilon>0$.

\section{Differentiating hyperreals $[f]$ such that $[f']$ exists}
\label{mainsection}

\begin{definition}
\label{idempotentultrafiltersonRdefn}
(Idempotent ultrafilters on $\mathbb R$)
    \begin{enumerate}
        \item
        For each $S\subseteq \mathbb R$ and any $y\in\mathbb R$,
        $S-y$ is defined to be $\{x-y\,:\,x\in S\}$.
        \item
        An ultrafilter $p\in\beta\mathbb R$ is \emph{idempotent} if
        $p=\{S\subseteq \mathbb R \,:\, \{y\in\mathbb R\,:\,S-y\in p\}\in p\}$.
    \end{enumerate}
\end{definition}

\begin{definition}
    By $0^+$ we mean the set of ultrafilters $p\in\beta\mathbb R$
    such that $p$ satisfies the following requirement.
    For every real $\epsilon>0$, the open interval $(0,\epsilon)\in p$.
\end{definition}

\begin{lemma}
\label{zeroplushasidempotentultrafilterlemma}
    $0^+$ contains an idempotent ultrafilter.
\end{lemma}

\begin{proof}
    By Lemma 13.29(a) and Theorem 13.31 of \cite{hindman2011algebra}.
\end{proof}

Clearly an ultrafilter in $0^+$ is free.
The following lemma illustrates the power of ultrafilters in $0^+$.

\begin{lemma}
\label{neatlemma}
    Let $p\in 0^+$.
    If $f:\mathbb R\to\mathbb R$ is continuous at $0$ then
    $\mathrm{st}([f]_p)=f(0)$.
\end{lemma}

\begin{proof}
    Let $\epsilon>0$ be real. By continuity of $f$ at $0$,
    $\exists\delta>0$ such that $|f(0)-f(x)|<\epsilon$ whenever
    $x\in(0,\delta)$.
    Since $p\in0^+$, $(0,\delta)\in p$.
    Thus, $f$ is within $\epsilon$ of $f(0)$ ultrafilter often,
    so $[f]_p$ is within $\epsilon$ of $f(0)$.
\end{proof}

For the rest of this section, we fix an idempotent $p\in 0^+$.
The following theorem shows that this suffices to make the
definition $[f]'=[f']$ well-defined if we restrict it to
functions such that $[f']$ exists. Note that since $p\in 0^+$,
the existence of $[f']$ is equivalent to the statement that for all
real $\epsilon>0$, there exists real $\delta\in (0,\epsilon)$ such that
$f'(\delta)$ exists.

\begin{theorem}
\label{firsttheorem}
    For all $f,g:\mathbb R\to\mathbb R$ such that $[f']$ and $[g']$ exist,
    if $[f]=[g]$ then $[f']=[g']$.
\end{theorem}

\begin{proof}
    Since $[f]=[g]$, there is some $S_0\in p$ such that
    $f=g$ on $S_0$. Let $S=S_0\cap \mathrm{dom}(f')\cap \mathrm{dom}(g')$.
    Existence of $[f']$ and $[g']$ means $\mathrm{dom}(f')\in p$
    and $\mathrm{dom}(g')\in p$, thus $S\in p$.
    Since $p$ is idempotent, $\{x\in\mathbb R\,:\,S-x\in p\}\in p$.
    Thus $S\cap \{x\in\mathbb R\,:\,S-x\in p\}\in p$.
    To show $[f']=[g']$, we will show that
    $f'=g'$ on $S\cap \{x\in\mathbb R\,:\,S-x\in p\}$.
    Let $x\in S\cap \{x\in\mathbb R\,:\,S-x\in p\}$.
    In particular, $S-x\in p$. We must show $f'(x)=g'(x)$.

    Claim: For all $h\in S-x$, $(f(x+h)-f(x))/h=(g(x+h)-g(x))/h$.
    Indeed, let $h\in S-x$. This means $h=y-x$ for some $y\in S$.
    Compute:
    \begin{align*}
        (f(x+h)-f(x))/h
            &= (f(x+y-x)-f(x))/h
                &\mbox{($h=y-x$)}\\
            &= (f(y)-f(x))/h
                &\mbox{(Algebra)}\\
            &= (g(y)-g(x))/h
                &\mbox{($f=g$ on $S$)}\\
            &= (g(x+y-x)-g(x))/h
                &\mbox{(Algebra)}\\
            &= (g(x+h)-g(x))/h,
                &\mbox{($h=y-x$)}
    \end{align*}
    proving the claim.

    Since $p\in0^+$ and $S-x\in p$,
    it follows that for all real $\epsilon>0$, $(S-x)\cap (0,\epsilon)\in p$,
    thus is nonempty.
    So $S-x$ contains $h$ arbitrarily near $0$.
    Since $x\in \mathrm{dom}(f')\cap \mathrm{dom}(g')$,
    $\lim_{h\to 0}(f(x+h)-f(x))/h$ and $\lim_{h\to 0}(g(x+h)-g(x))/h$
    exist. Since both limits exist and since there are $h$ arbitrarily near $0$
    such that $(f(x+h)-f(x))/h=(g(x+h)-g(x))/h$,
    the limits must be equal, that is, $f'(x)=g'(x)$.
\end{proof}

\begin{corollary}
\label{primewelldefdcorollary}
    Let
    $
        \mathcal D
        =
        \{
            [f] \,:\, \mbox{$f:\mathbb R\to\mathbb R$ and $[f']$ exists}
        \}.
    $
    The derivative operation $\bullet':\mathcal D\to {}^*\mathbb R$
    defined by $[f]'=[f']$ (for all $f$ such that $[f']$ exists) is well-defined.
\end{corollary}

We do not yet know whether $\mathcal D$ (from
Corollary \ref{primewelldefdcorollary}) is a proper subset of ${}^*\mathbb R$.
In other words: can every hyperreal be written in the form $[f]$ where $[f']$
exists? Does this depend on $p$?

If $\mathcal D$ is as in Corollary \ref{primewelldefdcorollary} then
it follows that the structure $(\mathcal D,1,\Omega,+,\cdot,\bullet')$ satisfies
every positive formula in the theory of elementary calculus functions
in the language $(1,\mathrm{id},+,\cdot,\bullet')$, where,
by \emph{positive formula},
we mean a formula that can be built up without using $\neg$ or $\not=$.
In particular this includes the Leibniz rule axiom, linearity,
and the power rule schema.
The structure also satisfies the nontriviality axiom
$\exists x\,\mbox{s.t.}\,x'\not=0$.
All this remains true if constant symbols for other individual functions
(such as $\sin$ and $\cos$), besides just the identity function,
are added to the language, interpreted in $\mathcal D$ by
the hyperreals represented thereby (such as $[\sin]$ and $[\cos]$), provided
those hyperreals' derivatives exist.

In terms of nonstandard extensions, Theorem \ref{firsttheorem}
says that there is a well-defined map which sends
every ${}^*f(\Omega)$ to ${}^*f'(\Omega)$.
For example, this map sends $e^\Omega+\Omega^3+\cos2\Omega$
to $e^\Omega+3\Omega^2-2\sin2\Omega$.

One might intuitively wonder whether $[f]'=0$ implies $f$ is constant
(at least ultrafilter often). The following proposition provides a counterexample.

\begin{proposition}
    There exists $f:\mathbb R\to\mathbb R$ such that $[f]'=0$
    but for every $r\in\mathbb R$, $\{x\in\mathbb R\,:\,f(x)=r\}\not\in p$.
\end{proposition}

\begin{proof}
    By Theorem 1.14 of \cite{bankston1979topological},
    there exist disjoint Cantor sets on $\mathbb R$.
    An ultrafilter cannot contain two disjoint sets, so
    there is some Cantor set $C$ on $\mathbb R$ with $C\not\in p$.
    Let $f$ be the devil's staircase based on $C$.
    Then $f'(x)=0$ for all $x\not\in C$ so $[f]'=0$,
    but $f$ is increasing, and is not flat on $(0,\epsilon)$ for
    any $\epsilon>0$, implying (since $p\in0^+$)
    that $\{x\in\mathbb R\,:\,f(x)=r\}\not\in p$ for all $r$.
\end{proof}

We have proven Corollary \ref{primewelldefdcorollary}
under the assumption that $p$ is idempotent and in $0^+$. Similar reasoning
would hold if $p$ were idempotent and in $0^-$ (i.e., if $p$ were required to
contain $(-\epsilon,0)$ for every positive real $\epsilon$). We currently do not
know whether Corollary \ref{primewelldefdcorollary} holds for any other
type of ultrafilter.

\subsection{Differential equations and the secant method}

Since $\bullet'$ takes (a subset of) ${}^*\mathbb R$ to ${}^*\mathbb R$,
one can attempt to solve (or approximately solve) differential
equations by using the secant method from numerical analysis,
which is traditionally only used to solve non-differential equations.
This is interesting because as far as we know, the secant method has not
previously been applicable to differential equations.
We illustrate this with an example in which the method finds a
correct solution in one step.

\begin{example}
    Solve the differential equation $y'-2x=0$ using the secant method,
    with initial guesses $y_0=x^2+x^3$ and $y_1=x^2-x^3$.
\end{example}

\textit{Solution.}
    Define $\alpha:{\subseteq}{}^*\mathbb R\to{}^*\mathbb R$ by
    $\alpha([f])=[f]'-2\Omega$ whenever $[f]'$ is defined.
    We desire a solution $[f]$ of the equation $\alpha([f])=0$.
    Since $\Omega=[x\mapsto x]$, it follows that $[f]'-2\Omega=[x\mapsto f'(x)-2x]$,
    so any such solution $[f]$ will yield a solution $y=f(x)$ to
    the differential equation $y'-2x=0$ (at least $p$-a.e.). Compute:
    \begin{align*}
        [f_0] &= \Omega^2+\Omega^3
            &\mbox{(initial guess $y_0=x^2+x^3$)}\\
        [f_1] &= \Omega^2-\Omega^3
            &\mbox{(initial guess $y_1=x^2-x^3$)}\\
        [f_2] &= [f_1] - \alpha([f_1])
            \frac{[f_1]-[f_0]}{\alpha([f_1])-\alpha([f_0])};
                &\mbox{(Secant method)}\\
        \alpha([f_1]) &= (\Omega^2-\Omega^3)'-2\Omega\\
            &= (2\Omega-3\Omega^2)-2\Omega;\\
        \alpha([f_0]) &= (\Omega^2+\Omega^3)'-2\Omega\\
            &= (2\Omega+3\Omega^2)-2\Omega;
    \end{align*}
    it follows that $[f_2]=\Omega^2$.
    This yields a solution $y=x^2$ to the original differential equation.
    \qed

We have not yet found any examples where this approach is more practical
than other approximate methods in differential equations. We hope that
either such examples can be found later, or, if not, that the lack of such
examples might provide insight into limitations of the secant method itself.
The point of this subsection is not so much to focus on the secant
method, but rather to illustrate the kind of things we hope might be
possible by numerically interpreting the theory of elementary calculus
functions.

\subsection{The well-definability of composition on a subset of the hyperreals}
\label{entirenumberssectn}

The work we have presented above is relevant to numerically modeling
subtheories of the theory of elementary calculus functions in a language
containing a unary function symbol $\bullet'$ for differentiation.
But one key axiom is missing from that theory, namely the chain rule,
since the chain rule also involves a binary function symbol $\circ$ for
composition. In this section, we introduce a subset of the hyperreals
suitable for numerically modeling subtheories of the theory of elementary
calculus functions in a language containing $\bullet'$ and $\circ$.
The following definition is motivated by the theory of complex analysis.

\begin{definition}
    (Entire numbers)
    \begin{enumerate}
        \item
        A function $f:\mathbb R\to\mathbb R$ is \emph{entire} if
        $f$ is infinitely differentiable at $0$ and $\forall x\in\mathbb R$,
        $f(x)=\sum_{k=0}^\infty f^{(k)}(0)x^k/k!$.
        \item
        A hyperreal number is \emph{entire} if it can be written as $[f]$
        for some entire $f:\mathbb R\to\mathbb R$.
    \end{enumerate}
\end{definition}

\begin{proposition}
\label{injectivityonentires}
    For all entire $f,g:\mathbb R\to\mathbb R$, if $[f]=[g]$ then $f=g$.
\end{proposition}

\begin{proof}
    By well-known results from analysis, $f$ and $g$ are infinitely
    differentiable everywhere, so $[f^{(k)}]$ and $[g^{(k)}]$ exist
    for all $k$.
    By repeated applications of Theorem \ref{firsttheorem},
    each $[f^{(k)}]=[g^{(k)}]$. By Lemma \ref{neatlemma},
    it follows that each $f^{(k)}(0)=g^{(k)}(0)$.
    Since $f$ and $g$ are entire, this implies $f=g$.
\end{proof}

In the same way that we can think of the differential equation
$2x+2yy'=0$ as describing the family of all circles centered at the origin,
we can think of the equation
$y=\sum_{k=0}^\infty y^{(k)}(0)x^k/k!$ as describing the
family of all entire functions. The latter is much worse behaved than the
former: no two circles centered at the origin ever intersect each other,
but for all $(x_0,y_0)\in\mathbb R^2$ there exist distinct entire functions
whose graphs intersect at $(x_0,y_0)$. In this sense, we can say that the
real plane contains no ``critical points'' of $2x+2yy'=0$, but that every
point of the real plane is a ``critical point'' of
$y=\sum_{k=0}^\infty y^{(k)}(0)x^k/k!$.
We can interpret Proposition \ref{injectivityonentires} as saying that in
the hyperreal plane, every point on the vertical line $x=\Omega$ is a
``non-critical point'' of the family of all entire functions, for
the proposition says that no two distinct entire function graphs intersect
anywhere on this vertical line.

\begin{corollary}
\label{compositioncorollary}
    The operation $\circ$ defined on the entire numbers by
    $[f]\circ[g]=[f\circ g]$ (whenever $f$ and $g$ are entire)
    is well-defined.
\end{corollary}

Clearly with $\circ$ defined as in Corollary \ref{compositioncorollary} and
with $\bullet'$ defined as in Corollary \ref{primewelldefdcorollary},
the entire numbers satisfy the chain rule axiom,
$\forall x\forall y\,(x\circ y)'=(x'\circ y)y'$.

\subsection{The approximately space-filling nature of differentiation}

Clearly $\bullet'$ is linear over $\mathbb R$, in the sense that for all
$[f],[g]\in{}^*\mathbb R$ and $\lambda,\mu\in\mathbb R$,
$(\lambda[f]+\mu[g])'=\lambda [f]'+\mu [g]'$ if the derivatives
in question exist.
So how badly behaved could the graph $y=x'$ be?
We will show that it is approximately space-filling,
in the following sense.

\begin{definition}
    A subset $C\subseteq ({}^*\mathbb R)^2$
    is \emph{approximately space-filling} if for all
    $(x_0,y_0)\in\mathbb R^2$, there exists some $(x,y)\in C$
    such that $\mathrm{st}(x)=x_0$ and $\mathrm{st}(y)=y_0$.
\end{definition}

It is not hard to find functions ${}^*\mathbb R\to{}^{*}\mathbb R$
which are linear over $\mathbb R$ and whose graphs are
approximately space-filling. For example, if we consider ${}^*\mathbb R$
as a vector space over $\mathbb R$ and let $\mathcal B$ be a basis for it with
an infinitesimal basis element $v$, then
the projection $\pi(\cdots+\lambda v+\cdots)=\lambda$ is one such function.
Nevertheless, we find it interesting (even if the proof is quite simple)
that our derivative operator also has these properties.

\begin{proposition}
\label{approximatelyspacefillingprop}
    The hyperreal graph $y=x'$, i.e., the set of all $([f],[g])\in({}^*\mathbb R)^2$
    such that $[g]=[f]'$, is approximately space-filling.
\end{proposition}

\begin{proof}
    For any $(x_0,y_0)\in\mathbb R^2$, let $f:\mathbb R\to\mathbb R$
    be continuously differentiable everywhere with $f(0)=x_0$ and $f'(0)=y_0$.
    By Lemma \ref{neatlemma}, $\mathrm{st}([f])=f(0)=x_0$ and
    $\mathrm{st}([f]')=\mathrm{st}([f'])=f'(0)=y_0$.
\end{proof}

Having established that the hyperreal graph $y=x'$ is approximately space-filling,
we will proceed to state two additional results about approximately space-filling
sets in general, which therefore apply in particular to said graph.

\begin{proposition}
\label{shadowprop}
    If $C\subseteq ({}^*\mathbb R)^2$ is approximately space-filling, then
    for any $X\subseteq \mathbb R^2$, there exists some
    $S\subseteq {}^*\mathbb R$ such that
    \[
        X = \{(\mathrm{st}(x),\mathrm{st}(y))\,:\,(x,y)\in C\mbox{ and }x\in S\}.
    \]
    In particular, for any $X\subseteq \mathbb R^2$, there exists some
    $S\subseteq {}^*\mathbb R$ such that
    \[
        X = \{(\mathrm{st}(x),\mathrm{st}(x'))\,:\,x\in S\}.
    \]
\end{proposition}

\begin{proof}
    Straightforward.
\end{proof}

\begin{proposition}
    Suppose $C\subseteq ({}^*\mathbb R)^2$ is approximately space-filling
    where $C$ is the graph (over ${}^*\mathbb R$) of $y=F(x)$ for some
    $F:{\subseteq}{}^*\mathbb R\to{}^*\mathbb R$.
    Let $X\subseteq\mathbb R^2$ be the graph of the equation $y=f(x)$
    for some everywhere-differentiable $f:\mathbb R\to\mathbb R$.
    If $S$ is as in Proposition \ref{shadowprop}, then
    $F|_S$ has the same slope as $y=f(x)$ in the following sense:
    for every $x\in S$, for every real $\epsilon>0$, there exists some
    real $\delta>0$ such that for all $h\in {}^*\mathbb R$, if $0<|h|<\delta$
    and $x+h\in S$ then $|f'(\mathrm{st}(x))-(F(x+h)-F(x))/h|<\epsilon$.
    In particular, this is true when $F=\bullet'$.
\end{proposition}

\begin{proof}
    Straightforward.
\end{proof}

\section{A variant of finite calculus using an idempotent ultrafilter on $\mathbb N$}
\label{finitecalculussection}

In so-called finite calculus, one considers the ``derivative''
$f(x+1)-f(x)$ of $f$, see Section 2.6 of \cite{concrete}.
In this section, we will investigate a variant of this derivative,
namely $\Delta f(x)={}^*f(x+\Omega)-f(x)$ where $\Omega=[n\mapsto n]$
is the canonical hyperreal in the hyperreals constructed using
an idempotent ultrafilter on $\mathbb N=\{1,2,\ldots\}$ (note that we
omit $0$ from $\mathbb N$).
We will show that this finite derivative $\Delta$ has unexpected
connections to the standard derivative, and that the equivalence class
(in an iterated ultrapower construction) of
$(\Delta f)|_\mathbb N$ is well-defined as a function of
$[f|_\mathbb N]$ (we elaborate what we mean by ``in an iterated ultrapower
construction'' in Remark \ref{iteratedultrafilterremark} below).

The following Definitions \ref{idempotentultrafilteronNdefn}
and \ref{definitionofhyperrealsfromN},
and Lemma \ref{idempotentultrafiltersonNexist} are
$\mathbb N$-focused
analogies of the $\mathbb R$-focused
Definitions \ref{idempotentultrafiltersonRdefn}
and \ref{hyperrealsdefn}, and Lemma \ref{zeroplushasidempotentultrafilterlemma}
above, respectively. We prefer this slight redundancy (instead of defining
everything in higher generality) for the sake of concreteness.

\begin{definition}
\label{idempotentultrafilteronNdefn}
    (Idempotent ultrafilters on $\mathbb N$)
    \begin{enumerate}
        \item
        If $S\subseteq \mathbb N$ and $n\in\mathbb N$, let
        $S-n=\{x-n\,:\,\mbox{$x\in S$ and $x-n \in\mathbb N$}\}$.
        \item
        An ultrafilter $q$ on $\mathbb N$ is \emph{idempotent} if
        $q=\{S\subseteq \mathbb N\,:\, \{n\in\mathbb N\,:\,S-n\in q\}\in q\}$.
    \end{enumerate}
\end{definition}

\begin{lemma}
\label{idempotentultrafiltersonNexist}
    Idempotent ultrafilters on $\mathbb N$ exist and are free.
\end{lemma}

\begin{proof}
    See \cite{hindman2011algebra}.
\end{proof}

\begin{definition}
\label{definitionofhyperrealsfromN}
    If $q$ is a free ultrafilter on $\mathbb N$, we write ${}^*\mathbb R_q$
    for the hyperreal numbers constructed using $q$ in the usual way, and
    for each $f:\mathbb N\to\mathbb R$ we write $[f]_q$ for the hyperreal
    represented by $f$.
    If $q$ is clear from context, we will write
    ${}^*\mathbb R$ for ${}^*\mathbb R_q$, $[f]$ for $[f]_q$,
    and ${}^*f$ for the nonstandard extension ${}^*f:{}^*\mathbb R\to{}^*\mathbb R$
    of $f:\mathbb R\to\mathbb R$.
\end{definition}

For the remainder of the paper, we fix an idempotent ultrafilter $q$ on $\mathbb N$.
Lemma \ref{bergelsonlemma} below will replace $0^+$.

\begin{definition}
\label{remainderdefn}
    For each $\gamma\in\mathbb R^+\backslash\mathbb Q$, we define
    $\RM_\gamma:\mathbb N\to (-\gamma/2,\gamma/2)$ as follows
    (we pronounce $\RM$ as ``remainder'').
    For each $n\in\mathbb N$, we define $\RM_\gamma(n)=n-k\gamma$ where
    $k\gamma$ is the closest integer multiple of $\gamma$ to $n$ (there is
    a unique such closest integer multiple of $\gamma$ because $\gamma$
    is irrational and $n\in\mathbb N$).
\end{definition}

Note that the following lemma depends on $q$ being idempotent.

\begin{lemma}
\label{bergelsonlemma}
    Let $\gamma\in\mathbb R^+\backslash \mathbb Q$.
    For every real $\epsilon>0$,
    $\{n\in\mathbb N\,:\,|\RM_\gamma(n)|<\epsilon\}\in q$.
\end{lemma}

\begin{proof}
    Follows from Theorem 7.2 of \cite{bergelson2010ultrafilters}.
\end{proof}

For the rest of the section, let $\Omega=[n\in\mathbb N\mapsto n]$ (the canonical
element of ${}^*\mathbb R\backslash \mathbb R$).

\begin{definition}
\label{finitederivdefn}
    (Finite derivative)
    For each $f:\mathbb R\to\mathbb R$, we define
    the finite derivative $\Delta f:\mathbb R\to{}^*\mathbb R$ by
    $\Delta f(x)={}^*f(x+\Omega)-f(x)$.
\end{definition}

The following theorem shows that, when restricted to $\gamma$-periodic functions
for a fixed irrational real number $\gamma>0$, $\Delta$ is a constant multiple of
$\bullet'$ (up to an infinitesimal error), at least where $f'$ exists.
And even when $f'$ does not exist, $\Delta$ (times the same constant multiple)
sometimes still provides information about the slope in question.

\begin{theorem}
\label{trickygammatheorem}
    Let $\gamma\in\mathbb R^+\backslash \mathbb Q$, and let
    $f:\mathbb R\to\mathbb R$ be $\gamma$-periodic. Let $x\in\mathbb R$.
    \begin{enumerate}
        \item
        If $f'(x)$ exists, then
        $\mathrm{st}(\Delta f(x)/[\RM_\gamma])=f'(x)$.
        \item
        If $f'(x)$ fails to exist because $\lim_{h\to 0}(f(x+h)-f(x))/h$
        diverges to $\infty$ (resp.\ $-\infty$) then $\Delta f(x)/[\RM_\gamma]$
        is infinite (resp.\ negative infinite).
        \item
        Let $g:\mathbb R\to\mathbb R$ be $\gamma$-periodic.
        If both $f'(x)$ and $g'(x)$ fail to exist because
        $\lim_{h\to 0}(f(x+h)-f(x))/h$ and $\lim_{h\to 0}(g(x+h)-g(x))/h$
        both diverge to $\infty$ but $\lim_{h\to 0}(f(x+h)-f(x))/h$ diverges
        to $\infty$ faster (i.e., there exists $\delta>0$ such that
        $(f(x+h)-f(x))/h>(g(x+h)-g(x))/h$ whenever $0<|h|<\delta$) then
        $\Delta f(x)/[\RM_\gamma]>\Delta g(x)/[\RM_\gamma]$. Similarly for $-\infty$.
    \end{enumerate}
\end{theorem}

\begin{proof}
    (1) Fix $x\in \mathbb R$ such that $f'(x)$ exists.
    Define $g:\mathbb N\to\mathbb R$ by
    $g(n)=(f(x+n)-f(x))/\RM_\gamma(n)$,
    then $[g]=\Delta f(x)/[\RM_\gamma]$.
    To show $\mathrm{st}([g])=f'(x)$,
    it suffices to show that for every real $\epsilon>0$,
    $\{n\in\mathbb N\,:\,|f'(x)-g(n)|<\epsilon\}\in q$.
    Fix $\epsilon$.
    By definition of $f'$, there is some $\delta>0$
    such that
    $\left|f'(x)-(f(x+h)-f(x))/h\right|<\epsilon$
    whenever $0<|h|<\delta$.
    Let $S=\{n\in\mathbb N\,:\, |\RM_\gamma(n)|<\delta\}$.
    By Lemma \ref{bergelsonlemma}, $S\in q$.
    We claim that $|f'(x)-g(n)|<\epsilon$ for all $n\in S$.
    Let $n\in S$. Then
    \begin{align*}
        |f'(x)-g(n)|
            &= |f'(x)-(f(x+n)-f(x))/\RM_\gamma(n)|
                &\mbox{(Def.\ of $g$)}\\
            &= |f'(x)
                -(f(x+\RM_\gamma(n))-f(x))/\RM_\gamma(n)|
                &\mbox{($f$ is $\gamma$-periodic)}\\
            &< \epsilon.
                &\mbox{($0<|\RM_\gamma(n)|<\delta$)}
    \end{align*}

    The proofs of (2) and (3) are similar to (and easier than) the proof of (1).
\end{proof}

In particular,
$\Delta(f\circ g)(x)/[\RM_\gamma]
\approx
(\Delta f(g(x))/[\RM_\gamma])(\Delta g(x)/[\RM_\gamma])$
(with infinitesimal error)
provided $f'(g(x))$ and $g'(x)$ exist and $g$ is $\gamma$-periodic;
thus $\Delta/[\RM_\gamma]$
satisfies a chain rule. Contrast this with Graham, Knuth and Patashnik's
claim that ``there's no corresponding chain rule of finite calculus''
\cite{concrete}.

Without the idempotency requirement, it would be possible to find
free $q\in\beta\mathbb N$ so as to falsify Theorem \ref{trickygammatheorem}.
For example, one could choose $q$ such that
$\{n\in\mathbb N\,:\,\RM_\gamma(n)\in (\frac\gamma3,\frac{2\gamma}3)\}\in q$.
Then for any $\gamma$-periodic $f:\mathbb R\to\mathbb R$
such that $f(0)=f'(0)=0$ and $f(x)=1$ for all $x\in (\frac\gamma3,\frac{2\gamma}3)$,
it would follow that $\Delta f(0)=1$, implying
$\frac3{2\gamma}<\Delta f(0)/[\RM_\gamma]< \frac3\gamma$,
making the conclusion of
Theorem \ref{trickygammatheorem} (part 1) impossible at $x=0$.

As in Section \ref{mainsection}, we would like to transform the
derivative operation $\Delta:\mathbb R^\mathbb R\to ({}^*\mathbb R)^{\mathbb R}$
into a well-defined derivative on hyperreals.
But since $\Delta f(x)$ is itself hyperreal,
we will need to be careful.

\begin{definition}
\label{hyperhyperrealsdefn}
    For every $f:\mathbb N\to {}^*\mathbb R$, we write
    $\llbracket f\rrbracket$ for the equivalence class of
    $f$ modulo the equivalence relation defined by declaring
    that $h,g:\mathbb N\to {}^*\mathbb R$ are equivalent
    iff $\{n\in\mathbb N\,:\,h(n)=g(n)\}\in q$.
    We identify each $r\in\mathbb R$ with $\llbracket n\mapsto r\rrbracket$.
\end{definition}

\begin{remark}
\label{iteratedultrafilterremark}
    Suppose $f:\mathbb N\to{}^*\mathbb R$.
    Using Proposition 6.5.2 of Chang and Keisler \cite{chang1990model},
    one can in fact think of $\llbracket f\rrbracket$ as being a
    hyperreal number, but in a different copy of the hyperreals.
    To be more precise, one can think of $\llbracket f\rrbracket$
    as a hyperreal in ${}^*\mathbb R_{q\times q}$, where $\times$ is the
    filter product defined shortly before Proposition 6.2.1
    of \cite{chang1990model}.
\end{remark}

The following theorem will allow us to well-define
$\Delta [f]=\llbracket\Delta f\rrbracket$.

\begin{theorem}
\label{secondmainthm}
    Let $f,g:\mathbb R\to\mathbb R$.
    If $[f|_\mathbb N]=[g|_\mathbb N]$ then
    $\llbracket\Delta f|_\mathbb N\rrbracket
    =\llbracket\Delta g|_\mathbb N\rrbracket$.
\end{theorem}

\begin{proof}
    Assume $[f|_\mathbb N]=[g|_\mathbb N]$.
    We must show $\{n\in\mathbb N\,:\,\Delta f(n)=\Delta g(n)\}\in q$.
    Since $[f|_\mathbb N]=[g|_\mathbb N]$, there is some $S\in q$ such that
    $f(n)=g(n)$ for all $n\in S$.
    Since $q$ is idempotent, $\{n\in\mathbb N\,:\,S-n\in q\}\in q$.
    Thus $S\cap \{n\in\mathbb N\,:\,S-n\in q\}\in q$.
    We claim $\Delta f(n)=\Delta g(n)$
    for all $n\in S\cap \{n\in\mathbb N\,:\,S-n\in q\}$.
    Fix any such $n$.
    By construction, $S-n\in q$.
    By a similar argument as in the proof of Theorem \ref{firsttheorem},
    for all $h\in S-n$,
    $f(n+h)-f(n)=g(n+h)-g(n)$.
    Thus $[h\mapsto f(n+h)-f(n)]=[h\mapsto g(n+h)-g(n)]$,
    in other words ${}^*f(n+\Omega)-f(n)={}^*g(n+\Omega)-g(n)$,
    in other words $\Delta f(n)=\Delta g(n)$.
\end{proof}

\begin{corollary}
\label{welldefdfinitederivcorollary}
    The finite derivative $\Delta:{}^*\mathbb R\to {}^{*}\mathbb R_{q\times q}$
    defined by $\Delta[f]=\llbracket \Delta f\rrbracket$ is well-defined.
\end{corollary}

One might intuitively expect that $\Delta[f]=0$ should imply that
$f$ is constant (at least ultrafilter often); we present
a counterexample disproving this intuition and replacing it with
a characterization related to periodicity.

\begin{definition}
    Let $f:\mathbb N\to\mathbb R$. We say $f$ is \emph{$q$-a.e.\ constant}
    if $\exists r\in \mathbb R$ such that $\{n\in\mathbb N\,:\,f(n)=r\}\in q$.
    We say $f$ is \emph{$q$-a.e.\ $\Omega$-periodic} if
    $\{n\in\mathbb N\,:\,{}^*f(n+\Omega)=f(n)\}\in q$.
\end{definition}

\begin{theorem}
\label{counterexampleflatthm}
    \begin{enumerate}
        \item For all $[f]\in{}^*\mathbb R$,
            $\Delta[f]=0$ iff $f$ is $q$-a.e.\ $\Omega$-periodic.
        \item For all $[f]\in{}^*\mathbb R$,
            if $f$ is $q$-a.e.\ constant then $\Delta[f]=0$.
        \item There exists $f:\mathbb N\to\mathbb R$ such that
            $\Delta[f]=0$ but $f$ is not $q$-a.e.\ constant.
    \end{enumerate}
\end{theorem}

\begin{proof}
    (1) and (2) are straightforward. For (3), let
    $\gamma\in\mathbb R^+\backslash \mathbb Q$ and
    define $g:(-\gamma/2,\gamma/2)\to\mathbb R$ as follows.
    For $x\not=0$,
    let $g(x)=\lfloor 1/|x|\rfloor$
    where $\lfloor \bullet\rfloor$ is the greatest integer function.
    Let $g(0)=0$.
    We claim $f=g\circ \RM_\gamma$ witnesses (3).
    By Lemma \ref{bergelsonlemma},
    $\{n\in\mathbb N\,:\,|\RM_\gamma(n)|<\epsilon\}\in q$
    for all real $\epsilon>0$. Since $\lim_{x\to 0}g(x)=\infty$,
    it follows that $[f]$ is infinite, thus $f$ is not
    $q$-a.e.\ constant.

    By Lemma \ref{bergelsonlemma},
    $S=\{n\in\mathbb N\,:\,|\RM_\gamma(n)|<\gamma/4\}\in q$.
    We claim $\Delta f(n)=0$ for all $n\in S$, whence
    $\Delta[f]=\llbracket \Delta f\rrbracket=0$.
    Let $n\in S$.
    Since $\gamma$ is irrational, it follows that $\RM_\gamma(n)$ is not
    one of the jump discontinuity points of $g$, thus
    $\exists \delta>0$ such that $g(x)=g(\RM_\gamma(n))$
    whenever $|x-\RM_\gamma(n)|<\delta$.
    By Lemma \ref{bergelsonlemma},
    $T=\{m\in\mathbb N\,:\,|\RM_\gamma(m)|<\min(\delta,\gamma/4)\}\in q$.
    We claim $f(n+m)-f(n)=0$ for all $m\in T$,
    which establishes $\Delta f(n)=0$.
    Let $m\in T$.
    Since $|\RM_\gamma(n)|<\gamma/4$ and $|\RM_\gamma(m)|<\gamma/4$,
    it follows that $\RM_\gamma(n+m)=\RM_\gamma(n)+\RM_\gamma(m)$.
    Thus $|\RM_\gamma(n+m)-\RM_\gamma(n)|=|\RM_\gamma(m)|<\delta$,
    so $f(n+m)=g(\RM_\gamma(n+m))=g(\RM_\gamma(n))=f(n)$ by
    choice of $\delta$.
\end{proof}

\subsection{An alternate proof and strengthening of Hindman's theorem}

The well-definedness in Corollary \ref{welldefdfinitederivcorollary}
can be used to prove
Hindman's theorem (Theorem \ref{hindmanthm} below).
Formally, our proof of Hindman's theorem is
basically identical to the usual proof, but informally it appears
different because all references to idempotent ultrafilters are hidden
underneath the innocent-looking fact that the derivative
of a constant function is zero. But the idempotency of the underlying
ultrafilter was used to prove that the derivative in question is well-defined.

\begin{lemma}
\label{hindmanlemma}
    Let $c\in\mathbb R$. If $f_1,\ldots,f_k:\mathbb N\to\mathbb R$
    are such that each $[f_i]=c$, then there exist arbitrarily large
    $n\in\mathbb N$ such that the following requirement holds.
    For each $i\in\{1,\ldots,k\}$, ${}^*f_i(n+\Omega)=f_i(n)=c$.
\end{lemma}

\begin{proof}
    By Corollary \ref{welldefdfinitederivcorollary}
    each $\Delta[f_i]=\Delta c=0$, thus
    $\llbracket n\mapsto {}^*f_i(n+\Omega)-f_i(n)\rrbracket=0$,
    i.e.\ $S_i=\{n\in\mathbb N\,:\,{}^*f_i(n+\Omega)=f_i(n)\}\in q$.
    Since each $[f_i]=c$, each $T_i=\{n\in\mathbb N\,:\,f_i(n)=c\}\in q$.
    Thus $\bigcap_i (S_i\cap T_i)\in q$. The elements thereof witness the lemma.
\end{proof}

\begin{theorem}
\label{hindmanthm}
    (Hindman's Theorem)
    If $f:\mathbb N\to\mathbb R$ has finite range,
    then there exists some $c\in \mathbb R$
    and some infinite $S\subseteq \mathbb N$
    such that for all finite nonempty $X\subseteq S$,
    $f(\sum X)=c$.
\end{theorem}

\begin{proof}
    By the maximality of
    ultrafilters, $[f]=c$ for some $c\in\mathbb R$.
    For every finite $X\subseteq \mathbb N$,
    define $f_X:\mathbb N\to\mathbb R$ by $f_X(n)=f((\sum X)+n)$
    (note that $f_\emptyset=f$).

    Inductively, suppose we have defined
    $n_1<\cdots<n_k$ (an empty list if $k=0$) such that:
    \begin{enumerate}
        \item
        For each nonempty $X\subseteq\{n_1,\ldots,n_k\}$,
        $f(\sum X)=c$.
        \item
        For each $X\subseteq\{n_1,\ldots,n_k\}$, $[f_X]=c$.
    \end{enumerate}
    By Lemma \ref{hindmanlemma}, pick $n_{k+1}>\max\{n_1,\ldots,n_k\}$
    such that
    \[
    \mbox{($*$) for each $X\subseteq \{n_1,\ldots,n_k\}$,
    ${}^*f_X(n_{k+1}+\Omega)=f_X(n_{k+1})=c$.}
    \]
    For each $X\subseteq\{n_1,\ldots,n_{k+1}\}$
    with $n_{k+1}\in X$, we have (by $*$)
    $f(\sum X)=f_{X\backslash \{n_{k+1}\}}(n_{k+1})=c$
    because $X\backslash\{n_{k+1}\}\subseteq\{n_1,\ldots,n_k\}$.
    And for each $X\subseteq\{n_1,\ldots,n_k\}$,
    since (by $*$) ${}^*f_X(n_{k+1}+\Omega)=c$, it follows that
    ${}^*f_{X\cup\{n_{k+1}\}}(\Omega)=c$, i.e.,
    $[f_{X\cup\{n_{k+1}\}}]=c$. So ${n_1<\cdots<n_{k+1}}$ also
    satisfy 1--2. By induction, we obtain $n_1<n_2<\cdots$
    with the above properties, which clearly proves the theorem.
\end{proof}

In the above proof, we proved more than was required. This leads to
the following strengthening of Hindman's theorem.

\begin{theorem}
    ($\Omega$ as universal Hindman number)
    If $f:\mathbb N\to\mathbb R$ has finite range,
    then there exists some $c\in \mathbb R$
    and some infinite $S\subseteq \mathbb N$
    such that for all finite nonempty $X\subseteq S\cup\{\Omega\}$,
    ${}^*f(\sum X)=c$.
\end{theorem}

\begin{proof}
    In the proof of Theorem \ref{hindmanthm}, we proved the existence
    of $c\in\mathbb R$ and $n_1,n_2,\ldots\in \mathbb N$ such that
    (1) for each finite nonempty
    $X\subseteq \{n_1,n_2,\ldots\}$, $f(\sum X)=c$, and (2) for each finite
    $X\subseteq \{n_1,n_2,\ldots\}$, $[f_X]=c$, where
    $f_X:\mathbb N\to\mathbb R$ is defined by $f_X(n)=f((\sum X)+n)$.
    For $X\subseteq\mathbb N$, the statement $[f_X]=c$ is equivalent
    to ${}^*f_X(\Omega)=c$,
    which is equivalent to ${}^*f(\sum(X\cup\{\Omega\}))=c$.
\end{proof}

In a heuristical sense, our proof of Theorem \ref{hindmanthm}
seems to suggest that the idempotency requirement might be indispensable for
Corollary \ref{welldefdfinitederivcorollary}: if the corollary
could be proven using weaker assumptions about $q$, then we would have a
non-idempotent ultrafilter proof of Hindman's theorem,
which seems like it would be surprising. But of course,
this is not a rigorous proof, and we do not actually know whether
there exists any non-idempotent ultrafilter for which
Corollary \ref{welldefdfinitederivcorollary} holds.

\subsection{Differentiating by $[\RM_\gamma]$}

In Theorem \ref{trickygammatheorem} we established a deep connection
between our finite derivative and the usual derivative from elementary
calculus. But the theorem was limited to $\gamma$-periodic functions for
some irrational $\gamma$. At first glance, this seems very limiting.
But since $\llbracket\Delta f\rrbracket$ only depends on $f|_\mathbb N$,
the following lemma shows that the $\gamma$-periodic hypothesis
in Theorem \ref{trickygammatheorem} does not limit the $\bullet'$-like
nature of the derivative from Corollary \ref{welldefdfinitederivcorollary}
at all.

\begin{lemma}
    Let $\gamma\in \mathbb R^+\backslash \mathbb Q$.
    For every $f:\mathbb R\to\mathbb R$, 
    there exists a $\gamma$-periodic function $\hat f:\mathbb R\to\mathbb R$
    such that $f|_\mathbb N=\hat f|_\mathbb N$.
\end{lemma}

\begin{proof}
    Define $\hat f:\mathbb R\to\mathbb R$ by
    \[
        \hat f(x) =
        \begin{cases}
            f(n) & \mbox{if $x=m\gamma+n$ for some $m\in\mathbb Z, n\in\mathbb N$,}\\
            0 &\mbox{in any other case.}
        \end{cases}
    \]
    Since $\gamma$ is irrational, there do not exist distinct ways to write
    $x=m\gamma+n$, thus $\hat f(x)$ is well-defined.
    Clearly $\hat f$ has the desired properties.
\end{proof}

Thus if $f:\mathbb R\to\mathbb R$
then
$\Delta[f|_\mathbb N]
=\Delta[\hat f|_\mathbb N]
=\llbracket n\mapsto \Delta \hat f(n)\rrbracket$
encodes (by Theorem \ref{trickygammatheorem})
information not about the derivative of $f$ but rather about the
derivative of $\hat f$. Since $[f|_\mathbb N]=[\hat f|_\mathbb N]$
this means that $\Delta[f|_\mathbb N]$
does indeed encode information about the hyperreal
$[f|_\mathbb N]$, just not necessarily about $f$.
For example, if $f(x)=x$ for all $x\in \mathbb R$,
then $\hat f$ is exotic and $\Delta[f|_\mathbb N]/[\RM_\gamma]$ is far from
the derivative $x'=1$ we might expect.
Nonetheless, we can use Theorem \ref{trickygammatheorem} to obtain
some other derivatives for which the familiar rules of elementary
calculus apply.

\begin{definition}
\label{Dgammadefn}
    For every $\gamma\in\mathbb R^+\backslash \mathbb Q$,
    for every $f:\mathbb N\to\mathbb R$,
    define $D_\gamma f:{\subseteq}\mathbb N\to\mathbb R$ by
    $D_\gamma f(n) = \mathrm{st}(\Delta f(n)/[\RM_\gamma])$
    for all $n$ such that $\mathrm{st}(\Delta f(n)/[\RM_\gamma])$
    is defined.
\end{definition}

\begin{definition}
\label{factorthroughstdefn}
    For every $\gamma\in\mathbb R^+\backslash\mathbb Q$,
    define $D_\gamma:{\subseteq}{}^*\mathbb R\to{}^*\mathbb R$
    as follows.
    For every $[f]\in{}^*\mathbb R$
    (so $f:\mathbb N\to\mathbb R$),
    define $D_\gamma[f]=[n\in\mathbb N\mapsto D_\gamma f(n)]$
    provided $[n\in\mathbb N\mapsto D_\gamma f(n)]$ is defined.
\end{definition}

\begin{proposition}
    Let $\gamma\in\mathbb R^+\backslash \mathbb Q$.
    The operator $D_\gamma:{\subseteq}{}^*\mathbb R\to{}^*\mathbb R$
    of Definition \ref{factorthroughstdefn} is well-defined.
\end{proposition}

\begin{proof}
    Suppose $f,g:\mathbb N\to\mathbb R$ are such that
    $[f]=[g]$.
    By Theorem \ref{secondmainthm},
    $\llbracket \Delta f\rrbracket = \llbracket \Delta g\rrbracket$.
    Thus, $\exists S\in q$ such that
    $\Delta f(n)=\Delta g(n)$ for all $n\in S$.
    Thus for all $n\in S$, $\Delta f(n)/[\RM_\gamma]=\Delta g(n)/[\RM_\gamma]$,
    and $\mathrm{st}(\Delta f(n)/[\RM_\gamma])$ is defined
    iff $\mathrm{st}(\Delta g(n)/[\RM_\gamma])$ is defined.
    Thus $D_\gamma [f]=D_\gamma[g]$ (and either both sides are
    defined, or both sides are undefined).
\end{proof}

\begin{theorem}
\label{dimensionstheorem}
    Let $\gamma\in\mathbb R^+\backslash\mathbb Q$.
    If $f:\mathbb R\to\mathbb R$ is differentiable on
    $(-\gamma/2,\gamma/2)$,
    then
    $D_\gamma [f\circ \RM_\gamma] = [f'\circ \RM_\gamma]$.
\end{theorem}

\begin{proof}
    Let $S=\mathbb R\backslash\{(k+\frac12)\gamma\,:\,k\in\mathbb Z\}$.
    Define $\overline{\RM}_\gamma$ the same way
    $\RM_\gamma$ was defined (Definition \ref{remainderdefn})
    except define it for all $x\in S$ instead of only $x\in\mathbb N$.
    Clearly $\overline{\RM}_\gamma|_\mathbb N=\RM_\gamma|_\mathbb N$ and
    $\overline{\RM}_\gamma$ is differentiable with $\overline{\RM}_\gamma'(x)=1$ on
    $S$.
    Since $f$ is differentiable on
    $(-\gamma/2,\gamma/2)= \overline{\RM}_\gamma(S)$,
    we can apply the chain rule and see
    $f(\overline{\RM}_\gamma(x))'
    =f'(\overline{\RM}_\gamma(x))\overline{\RM}_\gamma'(x)
    =f'(\overline{\RM}_\gamma(x))$
    on $S$ ($*$).
    For every $n\in\mathbb N$,
    \begin{align*}
        D_\gamma (f\circ\RM_\gamma)(n)
            &= \mathrm{st}(\Delta (f\circ\RM_\gamma)(n)/[\RM_\gamma])
                &\mbox{(Definition \ref{Dgammadefn})}\\
            &= \mathrm{st}(\Delta (f\circ\overline{\RM}_\gamma)(n)/[\RM_\gamma])
                &\mbox{($\RM_\gamma|_\mathbb N=\overline{\RM}_\gamma|_\mathbb N$)}\\
            &= (f\circ\overline{\RM}_\gamma)'(n)
                &\mbox{(Theorem \ref{trickygammatheorem})}\\
            &= f'(\overline{\RM}_\gamma(n))
                &\mbox{(By $*$)}\\
            &= f'(\RM_\gamma(n)),
                &\mbox{($\RM_\gamma|_\mathbb N=\overline{\RM}_\gamma|_\mathbb N$)}
    \end{align*}
    so $[n\mapsto D_\gamma (f\circ\RM_\gamma)(n)]=[n\mapsto f'(\RM_\gamma(n))]$,
    i.e., $D_\gamma [f\circ \RM_\gamma]=[f'\circ \RM_\gamma]$.
\end{proof}

For example,
\[
    D_\gamma(e^{[\RM_\gamma]} + [\RM_\gamma]^3 + \cos2[\RM_\gamma])
    =
    e^{[\RM_\gamma]} + 3[\RM_\gamma]^2 - 2\sin2[\RM_\gamma].
\]
Thus, $D_\gamma$
follows the familiar derivative rules from
elementary calculus as long as we consider functions not of the
continuous variable $x\in\mathbb R$, but rather of the
discrete variable $\RM_\gamma(n)\in\RM_\gamma(\mathbb N)$.
In a sense, one can ``differentiate by $[\RM_\gamma]$'';
it might even be tempting to write $D_\gamma$ as $d/d[\RM_\gamma]$.
This is spiritually similar to how the prime numbers play the role of
dimensions, and how one differentiates by prime numbers,
in Jeffries' paper \cite{jeffries1772differentiating} in the Notices.

\section*{Acknowledgments}

We gratefully acknowledge Arthur Paul Pedersen
and the reviewers and the editor for comments and feedback.

\bibliographystyle{jloganal}
\bibliography{main}

\begin{thebibliography}{}
\providecommand\bibmarginpar{\leavevmode\marginpar}
\def\urlstyle#1{{\tt #1}}

\bibitem{bankston1979topological}
\textbf{P Bankston}, \textbf{R\,J McGovern}, \href{http://dx.doi.org/https://doi.org/10.1016/0016-660X(79)90034-5} {\emph{Topological partitions}}, General Topology and its Applications 10 (1979) 215--229

\bibitem{barbeau1961remarks}
\textbf{E Barbeau}, \href{http://dx.doi.org/https://doi.org/10.4153/CMB-1961-013-0} {\emph{Remarks on an arithmetic derivative}}, Canadian Mathematical Bulletin 4 (1961) 117--122

\bibitem{bergelson2010ultrafilters}
\textbf{V Bergelson}, \href{http://dx.doi.org/https://doi.org/10.1090/conm/530/10439} {\emph{Ultrafilters, {I}{P} sets, dynamics, and combinatorial number theory}}, from: ``Ultrafilters across Mathematics'', American Mathematical Society (2010)  23--47

\bibitem{buium2005arithmetic}
\textbf{A Buium}, \href{http://dx.doi.org/https://doi.org/10.1090/surv/118} {\emph{Arithmetic differential equations}}, American Mathematical Society (2005)

\bibitem{buium2015differential}
\textbf{A Buium}, \href{http://dx.doi.org/https://doi.org/10.1017/CBO9781316106877.011} {\emph{Differential calculus with integers}}, from: ``Arithmetic and Geometry, London Mathematical Society Lecture Note Series'', Cambridge University Press (2015)  139--187

\bibitem{buium2023foundations}
\textbf{A Buium}, \href{http://dx.doi.org/https://doi.org/10.1090/surv/222} {\emph{Foundations of arithmetic differential geometry}}, American Mathematical Society (2023)

\bibitem{chang1990model}
\textbf{C Chang}, \textbf{H\,J Keisler}, \emph{Model Theory}, 3rd edition, New York, Elsevier (1990)

\bibitem{concrete}
\textbf{R\,L Graham}, \textbf{D\,E Knuth}, \textbf{O Patashnik}, \emph{Concrete Mathematics: A Foundation for Computer Science}, 2nd edition, Addison-Wesley (1994)

\bibitem{hindman2011algebra}
\textbf{N Hindman}, \textbf{D Strauss}, \emph{Algebra in the {S}tone-{\v{C}}ech compactification: theory and applications}, Walter de Gruyter (2011)

\bibitem{jeffries1772differentiating}
\textbf{J Jeffries}, \href{http://dx.doi.org/https://doi.org/10.1090/noti2833} {\emph{Differentiating by prime numbers}}, Notices of the AMS 70 (2023)

\bibitem{kovic2012arithmetic}
\textbf{J Kovic}, \emph{The arithmetic derivative and antiderivative}, Journal of Integer Sequences 15 (2012)

\bibitem{pasten2022arithmetic}
\textbf{H Pasten}, \href{http://dx.doi.org/https://doi.org/10.4153/S0008439521000990} {\emph{Arithmetic derivatives through geometry of numbers}}, Canadian Mathematical Bulletin 65 (2022) 906--923

\bibitem{shelly1911cuestion}
\textbf{J\,M Shelly}, \emph{Una cuesti{\'o}n de la teor{\'\i}a de los n{\'u}meros}, Asociaci{\'o}n espa{\~n}ola, Granada  (1911) 1--12

\bibitem{stay2005generalized}
\textbf{M Stay}, \href{http://dx.doi.org/https://doi.org/10.48550/arXiv.math/0508364} {\emph{Generalized number derivatives}}, arXiv preprint math/0508364  (2005)

\bibitem{tossavainen2024we}
\textbf{T Tossavainen}, \textbf{P Haukkanen}, \textbf{J\,K Merikoski}, \textbf{M Mattila}, \href{http://dx.doi.org/https://doi.org/10.1080/07468342.2023.2268494} {\emph{We Can Differentiate Numbers, Too}}, The College Mathematics Journal 55 (2024) 100--108

\bibitem{ufnarovski2003differentiate}
\textbf{V Ufnarovski}, \textbf{B {\accent23A}hlander}, \emph{How to differentiate a number}, Journal of Integer Sequences 6 (2003) 03--3

\end{thebibliography}

\end{document}